\documentclass[11pt,a4paper]{article}

\usepackage{epsf,epsfig,amsfonts,amsgen,amsmath,amstext,amsbsy,amsopn,amsthm,cases,listings,color
}
\usepackage{ebezier,eepic}
\usepackage{color}
\usepackage{multirow}
\usepackage{epstopdf}
\usepackage{graphicx}
\usepackage{pgf,tikz}
\usepackage{mathrsfs}
\usepackage[marginal]{footmisc}
\usepackage{enumitem}
\usepackage[titletoc]{appendix}
\usepackage{booktabs}
\usepackage{url}
\usepackage{mathtools}

\usepackage{pgfplots}
\usepackage{authblk}
\usepackage{amssymb}

\usepackage{subfigure}
\usepackage{mathrsfs}
\usetikzlibrary{arrows}

\allowdisplaybreaks[1]

\definecolor{uuuuuu}{rgb}{0.27,0.27,0.27}
\definecolor{sqsqsq}{rgb}{0.1255,0.1255,0.1255}

\newtheorem{definition}{Definition} [section]
\newtheorem{theorem}[definition]{Theorem}
\newtheorem{lemma}[definition]{Lemma}

\newtheorem{claim}[definition]{Claim}

\newtheorem{observation}[definition]{Observation}

\newenvironment{pf}{\noindent{\bf Proof}}

\pgfplotsset{compat=1.18}
\begin{document}
\title{\bf\Large Hypergraphs with irrational Tur\'{a}n density and many extremal configurations}

\date{\today}
\author[1]{Jianfeng Hou\thanks{Research was supported by National Natural Science Foundation of China (Grant No. 12071077). Email: jfhou@fzu.edu.cn}}
\author[2]{Heng Li\thanks{Email: heng.li@sdu.edu.cn}}
\author[2]{Guanghui Wang\thanks{Email: ghwang@sdu.edu.cn}}
\author[1]{Yixiao Zhang\thanks{Email: fzuzyx@gmail.com}}
\affil[1]{Center for Discrete Mathematics,
            Fuzhou University, Fujian, 350108, China}
\affil[2]{School of Mathematics, Shandong University, Shandong, 250100, China}
\maketitle
\begin{abstract}
Unlike graphs, determining   Tur\'{a}n densities  of hypergraphs is known to be notoriously hard in general. The essential  reason is that for many classical families of $r$-uniform  hypergraphs $\mathcal{F}$, there are perhaps many near-extremal  $\mathcal{M}_t$-free  configurations with very different structure. Such a phenomenon is called not stable, and  Liu and Mubayi gave a first not stable example.  Another perhaps reason is that little is known about the set consisting  of all possible Tur\'{a}n densities which has cardinality of the continuum. Let $t\ge 2$ be an integer. In this paper,   we construct a  finite family  $\mathcal{M}_t$ of 3-uniform hypergraphs such that the Tur\'{a}n density of $\mathcal{M}_t$ is  irrational,  and there are $t$ near-extremal $\mathcal{M}_t$-free configurations that are far from each other in edit-distance.  This is the first not stable example that has an irrational  Tur\'{a}n density. It also provides  a new phenomenon about feasible region functions. 
\end{abstract}
\section{Introduction}\label{SEC:Introduction}
For an integer  $r\ge 2$, an $r$-uniform hypergraph (henceforth $r$-graph) $\mathcal{H}$ is a collection of $r$-subsets of some finite set $V$. We identify  $\mathcal{H}$ with its edge set,
use  $v(\mathcal{H})$  to denote the size of $V$.
Given a family $\mathcal{F}$ of $r$-graphs we say $\mathcal{H}$ is $\mathcal{F}$-free
if it does not contain any member of $\mathcal{F}$ as a subgraph.
The {\em Tur\'{a}n number} $\mathrm{ex}(n,\mathcal{F})$ of $\mathcal{F}$ is the maximum
number of edges in an $\mathcal{F}$-free $r$-graph on $n$ vertices.
The {\em Tur\'{a}n density} $\pi(\mathcal{F} )$ of $\mathcal{F}$ is defined as
$\pi(\mathcal{F})=\lim_{n\to \infty}\mathrm{ex}(n,\mathcal{F})/ {n\choose r}$.
A family $\mathcal{F}$ is called \emph{nondegenerate} if $\pi(\mathcal{F}) >0$.

The study of $\mathrm{ex}(n,\mathcal{F})$ is perhaps the central topic in extremal combinatorics, and much is known  when $r=2$.
The well-known Tur\'{a}n's theorem states that for every integer $\ell \ge 2$ the Tur\'{a}n number $\mathrm{ex}(n,K_{\ell+1})$
is uniquely achieved by the balanced $\ell$-partite graph on $n$ vertices, which is called the Tur\'{a}n graph $T(n,\ell)$. For $r\ge 3$ determining $\pi(\mathcal{F})$ for a family $\mathcal{F}$ of $r$-graphs  is much more difficult and many basic problems are
wide open.  Even very simple forbidden hypergraphs turned out to be notoriously difficult. For example, the famous tetrahedron conjecture of Tur\'{a}n~\cite{TU41} from 1941 that $\pi(K_{4}^{3}) = 5/9$  is still open, where $K_{\ell}^r$ denotes the complete $r$-graph on $\ell$ vertices. The best upper bound  is $\pi(K_{4}^{3}) \le 0.561666$, which was obtained by Razborov  $\cite{RA10}$ using the Flag Algebra machinery.

For every integer $r\ge 2$, define
\begin{align*}
\Pi_{\mathrm{fin}}^{(r)} & := \left\{\pi(\mathcal{F}) \colon \text{$\mathcal{F}$ is a finite family of $r$-graphs} \right\}, \quad\text{and} \\
\Pi_{\infty}^{(r)} & := \left\{\pi(\mathcal{F}) \colon \text{$\mathcal{F}$ is a (possibly infinite) family of $r$-graphs} \right\}.
\end{align*}
The classical Erd\H{o}s--Stone--Simonovits theorem~\cite{ES66,ES46} shows that
\[
\Pi_{\mathrm{fin}}^{(2)}=\Pi_{\infty}^{(2)}=\left\{\frac{m-1}{m}: m=1,2,\ldots\right\}.
\]
Unfortunately, little is known about $\Pi_{\mathrm{fin}}^{(r)}$  and $\Pi_{\infty}^{(r)}$ for $r\ge 3$. This is perhaps a barrier in  determining Tur\'{a}n densities of hypergrahs.  Erd\H{o}s \cite{Erdos64} proved that $\Pi_{\infty}^{(r)} \cap (0, r!/r) =\emptyset.$ Brown and Simonovits \cite{BS84} proved that $\Pi_{\infty}^{(r)}$ lies in the closure of  $\Pi_{\mathrm{fin}}^{(r)}$. Chung and Graham \cite{FG20} conjectured that  $\Pi_{\mathrm{fin}}^{(r)}$ consists of rational numbers only. This  was disproved by Baber and Talbot \cite{RJ11} who discovered a family of only three forbidden 3-graphs whose Tur\'{a}n density is irrational. A breakthrough in this topic is given by  Pikhurko~\cite{PI14} who showed that $\Pi_{\infty}^{(r)}$ has cardinality of the continuum for every $r\ge 3$. Recently, Yan and Peng \cite{YP22} constructed a single 3-graph whose Tur\'{a}n density is  irrational. Moreover, some Tur\'{a}n densities of hyerpgraphs have an interesting property which is called non-jump.  Our results in this paper are not very related to this, so we omit the related definitions here. We refer the interested reader to e.g. \cite{Peng19,Peng07,Peng07II,Peng08,Peng08I,Peng23} for related results.

Given a nondegenerate family $\mathcal{F}$ of $r$-graphs, an  $\mathcal{F}$-free $r$-graph with $n$ vertices and $\mathrm{ex}(n,\mathcal{F})-o(n^r)$ edges is called a \emph{near-extremal configuration} of $\mathcal{F}$.  A fundamental barrier in   determining   $\pi(\mathcal{F})$ is the presence of many near-extremal configurations with very different structures.  A classical example is $K_{4}^{3}$.  Assuming that Tur\'{a}n's Tetrahedron conjecture is true, Kostochka~\cite{KO82}  showed that there are at least $2^{n-2}$ nonisomorphic near-extremal configurations on $3n$ vertices. Another example is the Erd\H{o}s--S\'{o}s conjecture on 3-graphs with bipartite links (e.g. see~\cite{FF84,LM22}). To study this new phenomenon, Mubayi~\cite{MU07} made the following definition.
\begin{definition}[$t$-stable]\label{DFN:t-stable}
Let $r \ge 2$ and $t \ge 1$ be integers.
A family $\mathcal{F}$ of $r$-graphs is $t$-stable if for every $m\in\mathbb{N}$
there exist $r$-graphs $\mathcal{G}_{1}(m),\ldots,\mathcal{G}_{t}(m)$ on $m$ vertices such that the following holds.
For every $\delta >0$ there exist $\epsilon > 0$ and $N_0$ such that for all $n \ge N_0$
if $\mathcal{H}$ is an $\mathcal{F}$-free $r$-graph on $n$ vertices with $|\mathcal{H}| > (1-\epsilon) \mathrm{ex}(n,\mathcal{F})$
then $\mathcal{H}$ can be transformed to some $\mathcal{G}_{i}(n)$ by adding and removing at most $\delta n^r$ edges. The stability number of $\mathcal{F}$, denote by $\xi(\mathcal{F})$, the minimum integer $t$ such that $\mathcal{F}$ is $t$-stable, and set $\xi(\mathcal{F}) = \infty$ if there is no such $t$.
Say $\mathcal{F}$ is stable if  $\xi(\mathcal{F}) =1$ .
\end{definition}

The classical Erd\H{o}s--Simonovits stability theorem~\cite{SI68} implies that every nondegenerate family of graphs is stable.
For hypergraphs there are many families (whose Tur\'{a}n densities are unknown) which are conjecturally not stable. Using Kostochka's constructions Liu and Mubayi~\cite{LM22} proved  that $\xi(K_{4}^3) = \infty$ (assuming Tur\'{a}n's conjecture).
In fact, families that are not stable and whose Tur\'{a}n densities can be determined were constructed only very recently.
In~\cite{LM22}, Liu and Mubayi constructed the first finite $2$-stable family of $3$-graphs.
Later in~\cite{LMR1}, together with Reiher, they further constructed the first finite $t$-stable family of $3$-graphs
for every integer $t\ge 3$. Hou, Li, Liu, Mubayi and Zhang \cite{hou2022hypergraphs}  constructed finite families $\mathcal{F}$ of 3-graphs with  $\xi(\mathcal{F})=\infty$, which answered two open problems posed by Liu,  Mubayi and Reiher~\cite{LMR1}. Very recently,  Liu and Pikhurko \cite{LP22} constructed a finite family  $\mathcal{F}$ of 3-graphs such that there are exponentially many near-extremal configurations and then  $\xi(\mathcal{F})=\infty$. We remark that all constructions mentioned above has rational Tur\'{a}n densities. Our main result in this work is to give a novel method to construct a finite family $\mathcal{F}_t$ such  ${\rm ex}(n, \mathcal{F}_t)$ is an irrational number and $\xi(\mathcal{F}_t) = t$ for every positive integer $t$.


Before state our main result, we need the following definitions.
An $r$-graph $\mathcal{H}$ is a \emph{blowup} of an $r$-graph $\mathcal{G}$ if there exists a
map $\psi\colon V(\mathcal{H}) \to V(\mathcal{G})$ so that $\psi(E) \in \mathcal{G}$ iff $E\in \mathcal{H}$.
We say $\mathcal{H}$ is \emph{$\mathcal{G}$-colorable}
if there exists a map $\phi\colon V(\mathcal{H}) \to V(\mathcal{G})$ so that $\phi(E)\in \mathcal{G}$ for all $E\in \mathcal{H}$,
and we call such a map $\phi$ a \emph{homomorphism} from $\mathcal{H}$ to $\mathcal{G}$.
In other words, $\mathcal{H}$ is $\mathcal{G}$-colorable if and only if $\mathcal{H}$ occurs as a subgraph in some blowup
of $\mathcal{G}$. The following is our main result.

\begin{theorem}\label{THM:main-sec1.1}
For every positive integer $t$ there exist  an irrational number $\lambda_t\in (0,1/6)$, $3$-graphs $\mathcal{G}_1, \ldots,\mathcal{G}_t$  with $v(\mathcal{G}_i)=3t+5$ for $i\in [t]$, and  a finite family $\mathcal{M}_t$ of $3$-graphs with the following properties.
\begin{enumerate}[label=(\alph*)]
  \item $\pi(\mathcal{M}_t)=6\lambda_t$.
  \item For every  $t\ge 1$ there exist constants $\epsilon = \epsilon(t)>0$
        and $N_0 = N_0(t)$ such that the following holds for every integer $n\ge N_0$.
        Every $n$-vertex $\mathcal{M}_t$-free $3$-graph with minimum degree at least $\left(3\lambda_t-\epsilon\right)n^2$
        is $\mathcal{G}_{i}$-colorable for some $i\in [t]$.    
  \item $\xi(\mathcal{M}_t)=t$. 
\end{enumerate}
\end{theorem}

This paper is organized as follows. In Section \ref{Preliminaries}, we describe notations and  terminologies, and determine the Lagrangian of the complete $3$-graph with one edge removed. We introduce a novel operation of $3$-graphs in Section \ref{Mix-crossed-blowup}, which plays a key role in our proof.  The construction of $\mathcal{M}_t$  and a proof of Theorem \ref{THM:main-sec1.1} (a) are included in Section \ref{Construction}. In Section \ref{stability}, we prove Theorem \ref{THM:main-sec1.1} (b) and (c).  The final section contains some concluding remarks.

\section{Preliminaries}\label{Preliminaries}

\subsection{Definitions and notations}
At first, we give some definitions and notations. For positive integers $n_1$ and $n_2$ with $n_1\leq n_2$, let $[n_1]\colon=\{1,\ldots, n_1\}$ and $[n_1, n_2]\colon =[n_2]\setminus [n_1-1]$. To improve readability, as it is standard in the literature, we will usually pretend that large numbers are integers to avoid using essentially irrelevant floor and ceiling symbols.

Let $\mathcal{H}$ be  an $r$-graph. The \emph{shadow} of $\mathcal{H}$ is defined as
\begin{align}
\partial\mathcal{H}
=
\left\{A\in \binom{V(\mathcal{H})}{r-1}\colon \text{there is } B\in \mathcal{H} \text{ such that }
	A\subseteq B\right\}. \notag
\end{align}
For a vertex $v\in V(\mathcal{H})$, the \emph{link} $L_{\mathcal{H}}(v)$ of $v$ in $\mathcal{H}$ is
\[
L_\mathcal{H}(v)=\left\{A \in \partial\mathcal{H}\colon A \cup \{v\}\in \mathcal{H}\right\}.
\]
The \emph{degree} of $v$ in $\mathcal{H}$ is $d_{\mathcal{H}}(v)=|L_\mathcal{H}(v)|$. Denote by $\delta(\mathcal{H})$ and $\Delta(\mathcal{H})$
the minimum degree and maximum degree of $\mathcal{H}$, respectively. The \emph{neighborhood} $N_\mathcal{H}(v)$ of $v$ is defined as
\[
N_{\mathcal{H}}(v)
= \left\{u\in V(\mathcal{H})\setminus\{v\}\colon
     \exists E\in\mathcal{H}\text{ such that }\{u,v\}\subset E\right\}.
\]

Let $\mathcal{H}$ be a $3$-graph. The \emph{neighborhood} of $\{u,v\}\subset V(\mathcal{H})$ in $\mathcal{H}$ is
$$N_{\mathcal{H}}(u,v)=\left\{w\in V(\mathcal{H})\colon \{u,v,w\}\in \mathcal{H}\right\}.$$
The {\em codegree} of  $\{u,v\}$ is $d_{\mathcal{H}}(u,v) = |N_{\mathcal{H}}(u,v)|$.
Denote by $\delta_2(\mathcal{H})$ and $\Delta_2(\mathcal{H})$
the minimum degree and maximum codegree of $\mathcal{H}$, respectively.
For a set $A \subseteq V(\mathcal{H})$ denote by $\mathcal{H}[A]$ the induced subgraph of $\mathcal{H}$ with $A$. 
A pair $\{v_1,v_2\}\subset V(\mathcal{H})$ is \emph{symmetric} in $\mathcal{H}$ if
$$L_{\mathcal{H}}(v_1) - v_2 = L_{\mathcal{H}}(v_2) - v_1.$$
We will omit the subscript $\mathcal{H}$ from our notations if it is clear from the context.

Let $\mathcal{G}$ be an $r$-graph with the vertex set $[m]$, and $V_1,\ldots, V_m$ be $m$ disjoint sets
(each $V_i$ is allowed to be empty).
The \emph{associated blowup} $\mathcal{G}[V_1,\ldots, V_m]$ of $\mathcal{G}$ is obtained from $\mathcal{G}$
by replacing each vertex $i$ with the set $V_i$ and replacing each edge $i_1\cdots i_r$
with the complete $r$-partite $r$-graph with parts $V_{i_1},\ldots,V_{i_r}$.

For two $r$-graphs $F$ and $\mathcal{H}$, recall that we say $f \colon V(F) \rightarrow V(\mathcal{H})$ is a homomorphism if $f(E) \in \mathcal{H}$ for all $E \in F $. We say $\mathcal{H}$ is \emph{$F$-hom-free}
if there is no homomorphism from $F$ to $\mathcal{H}$. This is equivalent to say that every blowup of $\mathcal{H}$ is $F$-free.
For a family $\mathcal{F}$ of $r$-graphs we say $\mathcal{H}$ is \emph{$\mathcal{F}$-hom-free}
if it is $F$-hom-free for all $F \in \mathcal{F}$. An $r$-graph $F$ is $2$-covered if every $\{u,v\} \subset V(F)$ is contained in some $E \in F$, and a family $\mathcal{F}$ is \emph{$2$-covered} if all $F \in \mathcal{F}$ are $2$-covered. The following observation is easy to follow. 

\begin{observation}\label{F-free=F-hom-free}
If a family $\mathcal{F}$ of $r$-graphs is $2$-covered, then any $r$-graph $\mathcal{H}$ is $\mathcal{F}$-free if and only if it is $\mathcal{F}$-hom-free.
\end{observation}

\subsection{Lagrangian}

Denote by $\Delta_{m-1}$ the standard $(m-1)$-dimensional simplex,
i.e.
\begin{align}
\Delta_{m-1} = \left\{(x_1,\ldots,x_m)\in [0,1]^m\colon x_1+\ldots+x_m= 1\right\}. \notag
\end{align}
For an $r$-graph $\mathcal{H}$ on $m$ vertices (let us assume for notational transparency that $V(\mathcal{H}) = [m]$) the \emph{Lagrangian polynomial}
of $\mathcal{H}$ is defined by
\begin{align}
p_{\mathcal{H}}(x_1,\ldots,x_m) = \sum_{E\in \mathcal{H}}\prod_{i\in E}x_i. \notag
\end{align}
The \emph{Lagrangian} of $\mathcal{H}$ is defined by
\[
\lambda(\mathcal{H}) = \max\{p_{\mathcal{H}}(x_1,\ldots,x_m): (x_1,\ldots,x_m)\in \Delta_{m-1}\}.
\]
Since $\Delta_{m-1}$ is compact, a well known theorem of Weierstra\ss\ implies that  $\lambda(\mathcal{H})$ is well-defined.
Let
\[
Z(\mathcal{H})=\{(x_1,\ldots,x_m)\in \Delta_{m-1}: p_{\mathcal{H}}(x_1,\ldots,x_m)=\lambda(\mathcal{H})\}.
\]

The following  lemma gives a relationship between $\lambda(\mathcal{H})$ and the maximum number of edges in a blowup of $\mathcal{H}$ (e.g. see Frankl and F\"uredi~\cite{FF89} or Keevash's survey~\cite[{Section~3}]{KE11}).

\begin{lemma}[Frankl and F\"uredi~\cite{FF89}]\label{LEMMA:|H|<=lambda-T-n^r}
Let $r\geq 2$ and $\mathcal{G}$ and $\mathcal{H}$ be two $r$-graphs. Suppose that $\mathcal{G}$ is a blowup of $\mathcal{H}$ with $v(\mathcal{G}) = n$. Then $|\mathcal{G}| \leq \lambda(\mathcal{H}) n^r$.
\end{lemma}

As mentioned in \cite{LMR2}, every blowup  $\mathcal{G}$ of an $r$-graph  $\mathcal{H}$ with $|\mathcal{G}|=\lambda(\mathcal{H})v(\mathcal{G})^3$ needs to be regular in the sense that all vertices have the same degree. In the converse direction, $\mathcal{H}$ can still have much larger Lagrangians than  $|\mathcal{G}|/v(\mathcal{G})^3$. For instance, the Lagrangian of the Fano plane is 1/27 but not 1/49.

Fixed a  positive integer $t$, let $K_{3(t+1)}^{3-}$ denote the 3-graph obtained by the complete 3-graph $K_{3(t+1)}^3$ defined on $[3t+3]$  minus the  edge $\{3t+1,3t+2, 3t+3\}$. The following lemma shows that the Lagrangian of $K_{3(t+1)}^{3-}$ is an irrational number by an easy calculation.

\begin{lemma}\label{LEMMA:lambda(K-)-is-irrational}
\[
\lambda(K_{3(t+1)}^{3-})=\frac{t-t^2-3t^3-t^4}{4}-\frac{(3t^3+6t^2-t)\sqrt{9t^2+18t-3}}{36}.
\]
\end{lemma}
\begin{proof}
Notice that
\begin{align}
&p_{K_{3(t+1)}^{3-}}(x_1,\ldots, x_{3t+3}) \notag \\
=& \sum_{\{i, j, k\}\in \binom{[3t+3]}{3}}x_ix_jx_k - x_{3t+1}x_{3t+2}x_{3t+3} \notag\\
=& \sum_{\{i, j, k\}\in \binom{[3t]}{3}}x_ix_jx_k+\left(x_{3t+1}+x_{3t+2}+x_{3t+3}\right)\sum_{\{i, j\}\in \binom{[3t]}{2}}x_ix_j \notag\\
+&(x_{3t+1}x_{3t+2}+x_{3t+1}x_{3t+3}+x_{3t+2}x_{3t+3})\sum_{j\in [3t]}x_j.\notag
\end{align}
Set  $a=\frac{\sum_{i\in [3t]}x_i}{3t}$ and $b=\frac{x_{3t+1}+x_{3t+2}+x_{3t+3}}{3}$. Then
\[
3ta+3b=1.
\]
It follows from the AM-GM inequality that
\begin{align}\label{ALIGN:p_K-<..}
p_{K_{3(t+1)}^{3-}}(x_1,\ldots, x_{3t+3})&\leq  \binom{3t}{3}a^3+3b\binom {3t}{2}a^2+3b^2\cdot 3ta 
\end{align}
where the equality holds if and only if $x_i=a$ for $i\in [3t]$ and $x_{3t+1}=x_{3t+2}=x_{3t+3}=b$. Substituting $s=3t$ and $b=\frac{1-sa}{3}$ into \eqref{ALIGN:p_K-<..} yields that
\begin{align*}
p_{K_{3(t+1)}^{3-}}(x_1,\ldots, x_{3t+3})&\leq  \binom{s}{3}a^3+(1-sb)\binom {s}{2}a^2+\frac{1}{3}s(1-sa)^2a \notag\\
&=\frac{s}{6}\left(2a^3-(s+3)a^2+2b\right).
\end{align*}

Let
\[
f(a)=2a^3-(s+3)a^2+2a,
\]
and
\[
a(s)=\frac{s+3-\sqrt{s^2+6s-3}}{6}.
\]
Then $0<a(s)<1/s$ and $f(b)$ attains its maximum when $a=a(s)$. By the above arguments, we know that
\begin{align}\label{lambda(K-)}
\lambda(K_{3(t+1)}^{3-})&=  \max\left\{\frac{s}{6}\left(2a^3-(s+3)a^2+2a\right): 0\le a\le \frac{1}{s}\right\} \notag\\
&=\frac{s}{6}f(a(s))\notag\\
&=\frac{27s-9s^2-9s^3-s^4+s(s^2+6s-3)\sqrt{s^2+6s-3}}{324}\notag\\
&=\frac{t-t^2-3t^3-t^4}{4}-\frac{(3t^3+6t^2-t)\sqrt{9t^2+18t-3}}{36}.
\end{align}
\end{proof}

\textbf{Remarks.}
\begin{itemize}
\item For every integer $t\ge 1$,  $\sqrt{9t^2+18t-3}$ is an irrational number. Thus, by Lemma \ref{LEMMA:lambda(K-)-is-irrational}, $\lambda(K_{3(t+1)}^{3-})$ is irrational.

\item Let $a=\frac{2}{3t+3+\sqrt{9t^2+18t-3}}$, $b=\frac{1}{3}-ta$, through the proof of Lemma \ref{LEMMA:lambda(K-)-is-irrational}, we have
\[
Z\left(K_{3(t+1)}^{3-}\right)=\left\{\left(\underbrace{a,\ldots,a}_{3t},b,b,b\right)\right\}.
\]
\end{itemize}

\section{Mix-crossed blowup}\label{Mix-crossed-blowup}

In order to construct finite families of 3-graphs with infinitely many extremal constructions, Hou, Li, Liu, Mubayi and Zhang \cite{hou2022hypergraphs} defined an operation on $3$-graphs. Motivated by it, we give the following operation, which plays a key role in our proof.

\begin{definition}[Mix-crossed blowup]
Let $a,b,k$ be positive integers with $k\geq a+b+1$ and let  $\mathcal{G}$ be a $3$-graph. Suppose that  $\{v_1, v_2\}\subset V(\mathcal{G})$ is a pair of vertices with $d(v_1, v_2)=k$. Fix an ordering of vertices in $N_{\mathcal{G}}(v_1, v_2)$, say $N_{\mathcal{G}}(v_1, v_2)=\{u_1, \ldots, u_k\}$. The $(a,b)$-mix-crossed blowup of $\mathcal{G}$ on $\{v_1, v_2\}$, denoted by $\mathcal{G}\boxtimes_{(a,b)} \{v_1, v_2\}$, is a $3$-graph constructed from $\mathcal{G}$ and defined as follows.
\begin{itemize}
  \item [(a)] Remove all edges containing the pair $\{v_1, v_2\}$ from $\mathcal{G}$,
  \item [(b)] add two new vertices $v_1'$ and $v_2'$, make $v_1'$ a clone of $v_1$ and $v_2'$ a clone of $v_2$,
  \item [(c)] for every $i\in [a]$ add the edge set $\{u_iv_1v_2,u_iv_1v_2',u_iv_1'v_2,u_iv_1'v_2'\}$,
                 for $i\in[a+1,a+b]$ add the edge set $\{u_iv_1v_1',u_iv_1v_2',u_iv_2v_1',u_iv_2v_2'\}$, and for $i\in[a+b+1, k]$ add the edge set $\{u_iv_1v_1',u_iv_1v_2,u_iv_1'v_2',u_kv_2v_2'\}$.
\end{itemize}
\end{definition}

\begin{figure}[htbp]
\centering
\tikzset{every picture/.style={line width=0.75pt}} 

\begin{tikzpicture}[x=0.9pt,y=0.9pt,yscale=-1,xscale=1]

\draw  [color={rgb, 255:red, 208; green, 2; blue, 27 }  ,draw opacity=1 ][fill={rgb, 255:red, 208; green, 2; blue, 27 }  ,fill opacity=1 ] (319.8,26.6) .. controls (319.8,24.61) and (321.41,23) .. (323.4,23) .. controls (325.39,23) and (327,24.61) .. (327,26.6) .. controls (327,28.59) and (325.39,30.2) .. (323.4,30.2) .. controls (321.41,30.2) and (319.8,28.59) .. (319.8,26.6) -- cycle ;
\draw  [color={rgb, 255:red, 248; green, 231; blue, 28 }  ,draw opacity=1 ][fill={rgb, 255:red, 248; green, 231; blue, 28 }  ,fill opacity=1 ] (350.8,26.6) .. controls (350.8,24.61) and (352.41,23) .. (354.4,23) .. controls (356.39,23) and (358,24.61) .. (358,26.6) .. controls (358,28.59) and (356.39,30.2) .. (354.4,30.2) .. controls (352.41,30.2) and (350.8,28.59) .. (350.8,26.6) -- cycle ;
\draw  [color={rgb, 255:red, 144; green, 19; blue, 254 }  ,draw opacity=1 ][fill={rgb, 255:red, 144; green, 19; blue, 254 }  ,fill opacity=1 ] (383.8,26.6) .. controls (383.8,24.61) and (385.41,23) .. (387.4,23) .. controls (389.39,23) and (391,24.61) .. (391,26.6) .. controls (391,28.59) and (389.39,30.2) .. (387.4,30.2) .. controls (385.41,30.2) and (383.8,28.59) .. (383.8,26.6) -- cycle ;
\draw  [color={rgb, 255:red, 144; green, 19; blue, 254 }  ,draw opacity=1 ][fill={rgb, 255:red, 144; green, 19; blue, 254 }  ,fill opacity=1 ] (416.8,26.6) .. controls (416.8,24.61) and (418.41,23) .. (420.4,23) .. controls (422.39,23) and (424,24.61) .. (424,26.6) .. controls (424,28.59) and (422.39,30.2) .. (420.4,30.2) .. controls (418.41,30.2) and (416.8,28.59) .. (416.8,26.6) -- cycle ;
\draw  [color={rgb, 255:red, 144; green, 19; blue, 254 }  ,draw opacity=1 ][fill={rgb, 255:red, 144; green, 19; blue, 254 }  ,fill opacity=1 ] (478.8,26.6) .. controls (478.8,24.61) and (480.41,23) .. (482.4,23) .. controls (484.39,23) and (486,24.61) .. (486,26.6) .. controls (486,28.59) and (484.39,30.2) .. (482.4,30.2) .. controls (480.41,30.2) and (478.8,28.59) .. (478.8,26.6) -- cycle ;
\draw [color={rgb, 255:red, 208; green, 2; blue, 27 }  ,draw opacity=1 ][line width=0.75]    (378,107) -- (445.73,106.2) ;
\draw [color={rgb, 255:red, 144; green, 19; blue, 254 }  ,draw opacity=1 ]   (376,110) -- (443.73,109.2) ;
\draw [color={rgb, 255:red, 208; green, 2; blue, 27 }  ,draw opacity=1 ]   (376,166) -- (443.73,165.2) ;
\draw [color={rgb, 255:red, 144; green, 19; blue, 254 }  ,draw opacity=1 ]   (376,169) -- (443.73,168.2) ;
\draw [color={rgb, 255:red, 248; green, 231; blue, 28 }  ,draw opacity=1 ]   (375.91,106.23) -- (376.82,168.96) ;
\draw [color={rgb, 255:red, 144; green, 19; blue, 254 }  ,draw opacity=1 ]   (372.91,106.24) -- (373.82,168.97) ;
\draw [color={rgb, 255:red, 248; green, 231; blue, 28 }  ,draw opacity=1 ]   (445.91,105.23) -- (446.82,170.96) ;
\draw [color={rgb, 255:red, 144; green, 19; blue, 254 }  ,draw opacity=1 ]   (442.91,105.24) -- (443.82,170.97) ;
\draw [color={rgb, 255:red, 248; green, 231; blue, 28 }  ,draw opacity=1 ]   (375.91,106.23) -- (445.73,165.65) ;
\draw [color={rgb, 255:red, 208; green, 2; blue, 27 }  ,draw opacity=1 ]   (374,109) -- (443.73,168.2) ;
\draw [color={rgb, 255:red, 248; green, 231; blue, 28 }  ,draw opacity=1 ]   (447.08,109.21) -- (377.06,167.37) ;
\draw [color={rgb, 255:red, 208; green, 2; blue, 27 }  ,draw opacity=1 ]   (445.73,106.2) -- (375.72,164.35) ;
\draw  [color={rgb, 255:red, 0; green, 0; blue, 0 }  ,draw opacity=1 ][fill={rgb, 255:red, 0; green, 0; blue, 0 }  ,fill opacity=1 ] (371.13,107.8) .. controls (371.13,105.81) and (372.75,104.2) .. (374.73,104.2) .. controls (376.72,104.2) and (378.33,105.81) .. (378.33,107.8) .. controls (378.33,109.79) and (376.72,111.4) .. (374.73,111.4) .. controls (372.75,111.4) and (371.13,109.79) .. (371.13,107.8) -- cycle ;
\draw  [color={rgb, 255:red, 0; green, 0; blue, 0 }  ,draw opacity=1 ][fill={rgb, 255:red, 0; green, 0; blue, 0 }  ,fill opacity=1 ] (441.13,167.2) .. controls (441.13,165.21) and (442.75,163.6) .. (444.73,163.6) .. controls (446.72,163.6) and (448.33,165.21) .. (448.33,167.2) .. controls (448.33,169.19) and (446.72,170.8) .. (444.73,170.8) .. controls (442.75,170.8) and (441.13,169.19) .. (441.13,167.2) -- cycle ;
\draw  [color={rgb, 255:red, 0; green, 0; blue, 0 }  ,draw opacity=1 ][fill={rgb, 255:red, 0; green, 0; blue, 0 }  ,fill opacity=1 ] (371.22,167.37) .. controls (371.22,165.38) and (372.83,163.77) .. (374.82,163.77) .. controls (376.81,163.77) and (378.42,165.38) .. (378.42,167.37) .. controls (378.42,169.36) and (376.81,170.97) .. (374.82,170.97) .. controls (372.83,170.97) and (371.22,169.36) .. (371.22,167.37) -- cycle ;
\draw  [color={rgb, 255:red, 0; green, 0; blue, 0 }  ,draw opacity=1 ][fill={rgb, 255:red, 0; green, 0; blue, 0 }  ,fill opacity=1 ] (440.13,107.8) .. controls (440.13,105.81) and (441.75,104.2) .. (443.73,104.2) .. controls (445.72,104.2) and (447.33,105.81) .. (447.33,107.8) .. controls (447.33,109.79) and (445.72,111.4) .. (443.73,111.4) .. controls (441.75,111.4) and (440.13,109.79) .. (440.13,107.8) -- cycle ;
\draw  [dash pattern={on 0.84pt off 2.51pt}]  (446,27) -- (458.37,26.9) ;

\draw (320,38) node [anchor=north west][inner sep=0.75pt]  [font=\tiny] [align=left] {$v_3$};
\draw (350,38) node [anchor=north west][inner sep=0.75pt]  [font=\tiny] [align=left] {$v_4$};
\draw (383,38) node [anchor=north west][inner sep=0.75pt]  [font=\tiny] [align=left] {$v_5$};
\draw (415,38) node [anchor=north west][inner sep=0.75pt]  [font=\tiny] [align=left] {$v_6$};
\draw (475,38) node [anchor=north west][inner sep=0.75pt]  [font=\tiny] [align=left] {$v_{3t+3}$};
\draw (360,100) node [anchor=north west][inner sep=0.75pt]  [font=\tiny] [align=left] {$v_1$};
\draw (360,170) node [anchor=north west][inner sep=0.75pt]  [font=\tiny] [align=left] {$v_1'$};
\draw (450,98) node [anchor=north west][inner sep=0.75pt]  [font=\tiny] [align=left] {$v_2$};
\draw (450,168) node [anchor=north west][inner sep=0.75pt]  [font=\tiny] [align=left] {$v'_2$};

\end{tikzpicture}
\caption{An example: $K_{3t+3}^3\boxtimes_{(1,1)}\{v_1,v_2\}$}
\end{figure}
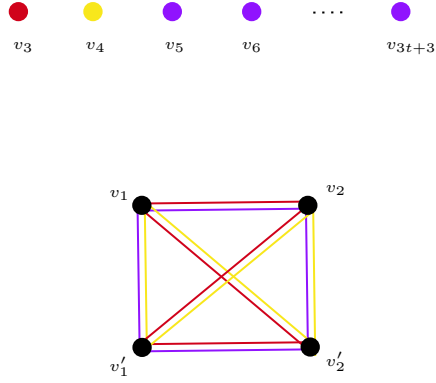

By the definition of the mix-crossed blowup, it is easy to check that
\begin{observation}\label{mix-crossed-blowup-2-covered}
The $3$-graph $\mathcal{G}\boxtimes_{(a,b)}\{v_1,v_2\}$ is $2$-covered iff $\mathcal{G}$ is $2$-covered.
\end{observation}

We end this section by the following key theorem.

\begin{theorem}\label{mix-crossed-blowup-lag}
Let  $\mathcal{G}$ be a $m$-vertex $3$-graph and $\{v_1,v_2\}\subset V(\mathcal{G})$ be a pair of vertices with $d(v_1,v_2)=k\ge a+b+1$ where $a>0, b>0$ be integers. Suppose that  $\{v_1,v_2\}$ is symmetric in $\mathcal{G}$. Then $\lambda(\mathcal{G}\boxtimes_{(a,b)}\{v_1,v_2\}) = \lambda(\mathcal{G})$. Moreover, if $(x_1,\ldots,x_m)\in Z(\mathcal{G})$, then
$\left((x_1+x_2)/4, (x_1+x_2)/4, (x_1+x_2)/4, (x_1+x_2)/4, x_3, \ldots, x_m\right)\in Z(\mathcal{G}\boxtimes_{(a,b)}\{v_1,v_2\}).$
\end{theorem}

\begin{proof}
For simplicity, assume that $V(\mathcal{G})=[m]$, $\{v_1,v_2\}=\{1,2\}$ and $N_{\mathcal{G}}(1, 2)=\{3, \ldots, k+2\}$. Let $X_1,\ldots, X_m$ be indeterminates. Since $\{1,2\}$ is symmetric,
\begin{align}
p_{\mathcal{G}}(X_1,\ldots,X_m) = \sum_{E\in \mathcal{H}}\prod_{i\in E}X_i=p_1+p_2(X_1+X_2)+p_3X_1X_2,
\end{align}
where $p_1=p_1(X_3,\ldots,X_m)=\sum_{E\in \mathcal{H}\setminus\{1,2\}}\prod_{i\in E}X_i$, $p_2=\sum_{i,j\in L_{\mathcal{G}}(1)\setminus \{2\}}X_iX_j$ and $p_3=\sum_{i=3}^{k+2}X_i$. For each  $(x_1,\ldots,x_m)\in Z(\mathcal{G})$, using the AM-GM inequality we obtain
\begin{align}
p_{\mathcal{G}}\left(\frac{x_1+x_2}{2}, \frac{x_1+x_2}{2},x_3\ldots, x_m\right)
& = p_1 + p_2\left(\frac{x_1+x_2}{2}+ \frac{x_i+x_j}{2}\right) + p_3\left(\frac{x_1+x_2}{2}\right)^2  \notag\\
& \ge p_1 + p_2(x_1+x_2) + p_3 x_1x_2
= \lambda(p). \notag
\end{align}
This means
\begin{align}\label{optimal-solution-G}
\left(\frac{x_1+x_2}{2}, \frac{x_1+x_2}{2},x_3\ldots, x_m\right)\in Z(\mathcal{G})
\end{align}

Considering the Lagrangian polynomial of $\mathcal{G}'=\mathcal{G}\boxtimes_{(a,b)}\{1,2\}$, we have
\begin{align*}
 p_{\mathcal{G}'}&(X_1,X'_1,X_2,X'_2,X_3,\ldots,X_m)  \notag\\
=& p_1 + p_2(X_1+X'_1+X_2+X'_2)
            + p_{4}(X_1+X'_1)(X_2+X'_2)
            + p_{5}(X_1+X_2)(X'_1+X'_2) \notag\\
&+ p_{6}(X_1+X'_2)(X'_1+X_2), \notag\\
\end{align*}
where $p_4=\sum_{i=3}^{a+2}X_i$,  $p_5=\sum_{i=a+3}^{a+b+2}X_i$ and $p_6=\sum_{i=a+b+3}^{k+2}X_i$. Then $p_3=p_4+p_5+p_6$. Let $\widehat{X}_1 = X_1+ X_1'$ and $\widehat{X}_2 = X_2+X_2'$. Then  using the AM-GM inequality again yields that
\begin{align}\label{equ:hat-p-and-p}
 p_{\mathcal{G}'}&(X_1,X'_1,X_2,X'_2,X_3,\ldots,X_m)  \notag\\
\le& p_1 + p_2(\widehat{X}_1+\widehat{X}_2)
            + p_{4}\left(\frac{\widehat{X}_1+\widehat{X}_2}{2}\right)^2
            + p_{5}\left(\frac{\widehat{X}_1+\widehat{X}_2}{2}\right)^2
            + p_6\left(\frac{\widehat{X}_1+\widehat{X}_2}{2}\right)^2\notag\\
=& p_1 + p_2(\widehat{X}_1+\widehat{X}_2)
            + p_{3}\left(\frac{\widehat{X}_1+\widehat{X}_2}{2}\right)^2 \notag\\
=& p_{\mathcal{G}}((\widehat{X}_1+\widehat{X}_2)/2, (\widehat{X}_1+\widehat{X}_2)/2, X_3, \ldots, X_m),
\end{align}
where the equality holds iff $X_1=X_1'=X_2=X_2'$.

Combing \eqref{optimal-solution-G} and \eqref{equ:hat-p-and-p}, we have that $\lambda(\mathcal{G}') = \lambda(\mathcal{G})$, and if $(x_1,\ldots,x_m)\in Z(\mathcal{G})$, then
\[
\left((x_1+x_2)/4, (x_1+x_2)/4, (x_1+x_2)/4, (x_1+x_2)/4, x_3, \ldots, x_m\right)\in Z(\mathcal{G}').
\]
This completes the proof of Theorem \ref{mix-crossed-blowup-lag}.
\end{proof}


%
%

\section{Construction  of $\mathcal{M}_t$ and  proof of its Tur\'{a}n density}\label{Construction}


In this section, we give a  construction  of $\mathcal{M}_t$ and prove its Tur\'{a}n density. For  a positive integer $t$, recall that  $K_{3t+1}^{3-}$ is the 3-graph obtained by the complete 3-graph $K_{3t+3}^3$ minus the  edge $\{3t+1, 3t+2, 3t+3\}$. For $i\in [t]$, let 
\[
\mathcal{G}_i=K_{3t+3}^{3-}\boxtimes_{(1, i)}\{1, 2\}.
\] 
By Observation \ref{mix-crossed-blowup-2-covered} and Theorem \ref{mix-crossed-blowup-2-covered}, we have

\begin{lemma}\label{Lemma:Gi-2-covered-lag}
For $i\in [t]$, $\mathcal{G}_i$ is $2$-covered and $\lambda(\mathcal{G}_i)=\lambda(K_{3t+3}^{3-})$.
\end{lemma}

For $n\ge 3t+5$, let $\mathcal{G}_n^i$ be a 3-graph on $n$ vertices which is a blowup of $\mathcal{G}_i$ with the maximum number of edges. Then by Lemma \ref{Lemma:Gi-2-covered-lag},
\begin{align}\label{number-edge-Gni}
|\mathcal{G}_n^i|=(1+o(1))\lambda(K_{3t+3}^{3-})n^3
\end{align}
For $i\in [t]$, let $\mathcal{F}_{\infty}(\mathcal{G}_n^i)$ be the (infinite) family of all $3$-graphs $F$ such that $\mathcal{G}_n^i$ is $F$-free for all positive integers $n$, i.e.,
\begin{align}
\mathcal{F}_{\infty}(\mathcal{G}_n^i) = \left\{r{\rm{\text{-}graph}}\ F \colon \mathcal{G}_n^i\ {\rm{is}}\ F{\rm{\text{-}free}}\ {\rm{for\ all\ positive\ integers}}\  n\right\}. \notag
\end{align}
Let
\begin{align}
\mathcal{F}_{m}(\mathcal{G}_n^i) = \left\{F \in \mathcal{F}_{\infty}(\mathcal{G}_n^i) \colon v(F) \leq m\right\}, \notag
\end{align}
and set
\begin{align}\label{M-choice}
\mathcal{M}_t =\bigcap_{i\in [t]} \mathcal{F}_{(3t+6)^2}(\mathcal{G}_n^i).
\end{align}

\textbf{Remark.}
\begin{itemize}
\item Let $\mathcal{K}_{\ell}^3$ be the family  $3$-graphs of $F$ with at most $\binom{\ell}{2}$ edges such that for some $\ell$-set $S$ (called 2-covered set) every pair $\{x, y\}\subset S$ is covered by an edge in $F$. For $F\in \mathcal{K}_{3t+6}^3$, $F$ has a 2-covered set of size $3t+6$. On the other hand, each $\mathcal{G}_n^i$ is $(3t+5)$-partite. So, $F\in \mathcal{F}_{\infty}(\mathcal{G}_n^i)$, which together with $v(F)\le (3t+6)+\binom{3t+6}{2}$ yields that $F\in \mathcal{M}_t$.
\end{itemize}

In  the following of this section, we determine the Tur\'{a}n density of $\mathcal{M}_t$ using the symmetrization method. Let $\mathcal{H}$ be an $r$-graph and $\{u,v\}\subset V(\mathcal{H})$ be two non-adjacent vertices (i.e., no edge contains both $u$ and $v$). We say $u$ and $v$ are \emph{equivalent} if $L_{\mathcal{H}}(u)=L_{\mathcal{H}}(v)$. Otherwise we say they are \emph{non-equivalent}. An equivalence class of $\mathcal{H}$ is a maximal vertex set in which every pair of vertices are equivalent. We say $\mathcal{H}$ is \emph{symmetrized} if for any two non-equivalent vertices $u$, $v \in V(\mathcal{H})$ there is an edge of $\mathcal{H}$ containing both of them.

\begin{definition}[Blowup-invariant]
A family $\mathcal{F}$ of $r$-graphs is blowup-invariant if every $\mathcal{F}$-free $r$-graph is also $\mathcal{F}$-hom-free.
\end{definition}

The following simple lemma is a special case of Lemma~8~\cite{PI14}.
\begin{lemma}[Pikhurko \cite{PI14}]\label{Lemma:blowup in}
The family $\mathcal{M}_t=\bigcap_{i\in [t]}\mathcal{F}_{(3t+3)^2}(\mathcal{G}_n^i)$ is blowup-invariant.
\end{lemma}

In~\cite{LMR2}, Liu, Mubayi and Reiher summarized the well known method of Zykov~\cite{Zy} symmetrization for solving Tur\'{a}n problems
into the following statement.

\begin{theorem}[Liu, Mubayi and Reiher \cite{LMR2}]\label{THEOREM:ex(n,F)=max-blowup-of-H}
Suppose that $\mathcal{F}$ is a blowup-invariant family of $r$-graphs. If $\mathfrak{H}$ denotes the class of all symmetrized $\mathcal{F}$-free $r$-graphs, then ex$(n, \mathcal{F})=\mathfrak{h}(n)$  holds for every $n \in \mathbb{N}^{+}$,
where $\mathfrak{h}(n)=\max\{|\mathcal{H}|\colon\mathcal{H}\in \mathfrak{H} \text{ and } v(\mathcal{H})=n\}$.
\end{theorem}

Now prove Theorem \ref{THM:main-sec1.1} (a) by showing

\begin{theorem}\label{turan-number-M}
Let $\mathcal{M}_t=\bigcap_{i\in [t]}\mathcal{F}_{(3t+3)^2}(\mathcal{G}_n^i)$. Then  $\pi(\mathcal{M}_t)=6\lambda(K_{3(t+1)}^{3-})$.
\end{theorem}
\begin{proof}
Note that each $\mathcal{G}_n^i$ is  $\mathcal{M}_t$-free. By \eqref{number-edge-Gni}, we have  $\pi(\mathcal{M})\ge6\lambda(K_{3(t+1)}^{3-})$. Now we prove the other side. Let  $\mathcal{H}$ is a  symmetrized $\mathcal{M}_t$-free $3$-graph on $n$ vertices. Combining Lemma \ref{Lemma:blowup in} and Theorem~\ref{THEOREM:ex(n,F)=max-blowup-of-H}, it suffices to prove that $|\mathcal{H}| \leq \lambda(K_{3(t+1)}^{3-})n^3$. Let $T \subseteq V(\mathcal{H})$ be the set that contains exactly one vertex from each equivalence class of $\mathcal{H}$ and let $\mathcal{T}=\mathcal{H}[T]$. Then $\mathcal{T}$ is 2-covered in  $\mathcal{H}$ and  $\mathcal{H}$ is a blowup of $\mathcal{T}$.   Note that $|T|\le 3t+5$, since otherwise $\mathcal{H}$  contains a copy of $F$ for some $F\in \mathcal{K}_{3t+6}^3$ whose  2-covered set is  contained in $T$, a contradiction.  Since $\mathcal{T}$ is 2-covered  and $\mathcal{H}$ does not contain any member of $\mathcal{M}_t$ as a subgraph, $\mathcal{T} \subseteq \mathcal{G}_{i_0}$  for some ${i_0}\in [t]$. It follows from Lemma~\ref{Lemma:Gi-2-covered-lag} that $\lambda(\mathcal{T}) \leq \lambda(\mathcal{G}_{i_0})=\lambda(K_{3t+3}^{3-})$. Thus, by Lemma \ref{LEMMA:|H|<=lambda-T-n^r}, $|\mathcal{H}|\le \lambda(\mathcal{T})n^3\le \lambda(K_{3(t+1)}^{3-})n^3$.
\end{proof}

\section{Stability}\label{stability}

In this section we always suppose that $\mathcal{M}_t =\bigcap_{i\in [t]} \mathcal{F}_{(3t+6)^2}(\mathcal{G}_n^i)$. First, we prove Theorem \ref{THM:main-sec1.1} (b) using the machinery provided in~\cite{LMR2}.

\begin{definition}[Symmetrized-stability]\label{DFN:sym-stability}
Let $\mathcal{F}$ be a family of $r$-graphs and let $\mathfrak{H}$ be a class of $\mathcal{F}$-free $r$-graphs. We say that $\mathcal{F}$ is symmetrized-stable with respect to $\mathfrak{H}$ if there exist $\epsilon > 0$ and $N_0$ such that every symmetrized $\mathcal{F}$-free $r$-graphs $\mathcal{H}$ on $n \geq N_0$ vertices with $\delta(\mathcal{H})\ge \bigl(\pi(\mathcal{F})/(r-1)!-\epsilon\bigr)n^{r-1}$ is a subgraph of a member of $\mathfrak{H}$.
\end{definition}

\begin{definition}[Vertex-extendibility]\label{DFN:vertex-extendable}
Let $\mathcal{F}$ be a family of $r$-graphs and let $\mathfrak{H}$ be a class of $\mathcal{F}$-free $r$-graphs.
We say that $\mathcal{F}$ is \emph{vertex-extendable}
with respect to $\mathfrak{H}$ if there exist $\zeta>0$ and $N_0\in\mathbb{N}$ such that
for every $\mathcal{F}$-free $r$-graph $\mathcal{H}$ on $n\ge N_0$
vertices satisfying $\delta(\mathcal{H})\ge \bigl(\pi(\mathcal{F})/(r-1)!-\zeta\bigr)n^{r-1}$
the following holds: if $\mathcal{H}-v$ is a subgraph of a member of $\mathfrak{H}$ for some
vertex $v\in V(\mathcal{H})$, then $\mathcal{H}$ is a subgraph of a member of $\mathfrak{H}$ as well.
\end{definition}

In~\cite{LMR2}, Liu, Mubayi and Reiher developed a machinery that reduces the proof of stability of certain families $\mathcal{F}$ to the simpler question of checking that an $\mathcal{F}$-free hypergraph $\mathcal{H}$ with large minimum degree is vertex-extendable.

\begin{theorem}[Liu, Mubayi and Reiher \cite{LMR2}]\label{THM:Psi-trick:G-extendable-implies-degree-stability}
Suppose that $\mathcal{F}$ is a blowup-invariant nondegenerate family of $r$-graphs and that
$\mathfrak{H}$ is a hereditary class of $\mathcal{F}$-free $r$-graphs.
If $\mathcal{F}$ is symmetrized-stable and vertex-extendable with respect to $\mathfrak{H}$, then the following statement holds.
There exist $\varepsilon>0$ and $N_0$ such that every $\mathcal{F}$-free $r$-graph on $n\ge N_0$ vertices with minimum
degree at least $\left(\pi(\mathcal{F})/(r-1)!-\varepsilon\right)n^{r-1}$ is contained in $\mathfrak{H}$.
\end{theorem}

In order to find some specific structures,  we need the following lemma given by Liu, Mubayi and Reiher \cite{LMR1}.
\begin{lemma}[\cite{LMR1}]\label{LEMMA:greedily-embedding-Gi}
Fix a real $\eta \in (0, 1)$ and integers $m, n\ge 1$.
Let $\mathcal{G}$ be a $3$-graph with the vertex set~$[m]$ and let $\mathcal{H}$ be a further $3$-graph
with $v(\mathcal{H})=n$.
Consider a vertex partition $V(\mathcal{H}) = \bigcup_{i\in[m]}V_i$ and the associated
blow-up $\widehat{\mathcal{G}} = \mathcal{G}[V_1,\ldots,V_{m}]$ of $\mathcal{G}$.
If two sets $T \subseteq [m]$ and $S\subset V$
have the properties
\begin{enumerate}[label=(\alph*)]
\item\label{it:47a} $|V_{j}| \ge (|S|+1)|T|\eta^{1/3} n$  for all $j \in T$,
\item\label{it:47b} $|\mathcal{H}[V_{j_1},V_{j_2},V_{j_3}]| \ge |\widehat{\mathcal{G}}[V_{j_1},V_{j_2},V_{j_3}]|
		- \eta n^3$ for all $\{j_1,j_2,j_3\} \in \binom{T}{3}$, and
\item\label{it:47c} $|L_{\mathcal{H}}(v)[V_{j_1},V_{j_2}]| \ge |L_{\widehat{\mathcal{G}}}(v)[V_{j_1},V_{j_2}]|
		- \eta n^2$ for all $v\in S$ and $\{j_1,j_2\} \in \binom{T}{2}$,
\end{enumerate}
then there exists a selection of vertices $u_j\in V_j\setminus S$ for all $j\in [T]$
such that $U = \{u_j\colon j\in T\}$ satisfies
$\widehat{\mathcal{G}}[U] \subseteq \mathcal{H}[U]$ and
$L_{\widehat{\mathcal{G}}}(v)[U] \subseteq L_{\mathcal{H}}(v)[U]$ for all $v\in S$. In particular, if $\mathcal{H} \subseteq \widehat{\mathcal{G}}$,
then $\widehat{\mathcal{G}}[U] = \mathcal{H}[U]$ and
$L_{\widehat{\mathcal{G}}}(v)[U] = L_{\mathcal{H}}(v)[U]$ for all $v\in S$.
\end{lemma}

\begin{proof}[Proof of Theorem \ref{THM:main-sec1.1} (b)]
Let
\begin{align}\label{frak-H}
\mathfrak{H}=\{\text{3-graph }\mathcal{H}: \mathcal{H} \text{ is } \mathcal{G}_i \text{-colorable}  \text{ for some } i\in [t]\}. 
\end{align}

By Theorem \ref{THM:Psi-trick:G-extendable-implies-degree-stability}, it suffices to show that $\mathcal{M}_t$ is symmetrized-stable and vertex-extendable with respect to $\mathfrak{H}$. Suppose that  $\mathcal{H}$ is a symmetrized $\mathcal{M}_t$-free $3$-graphs  on $n$ vertices. Let $T \subseteq V(\mathcal{H})$ be a set that contains exactly one vertex from each equivalence class of $\mathcal{H}$ and $\mathcal{T}=\mathcal{H}[T]$. Then  $\mathcal{H}$ is a blowup of $\mathcal{T}$. As in the proof of Theorem \ref{turan-number-M}, if $|T|> 3t+5$, then $\mathcal{H}$ contains a copy of $F\in \mathcal{K}_{3t+6}^3$, a contradiction. Otherwise, $\mathcal{T} \subseteq \mathcal{G}_{i_0}$ for some $i_0\in [t]$. This means $\mathcal{H}\in \mathfrak{H}$. Thus,  $\mathcal{M}_t$ is symmetrized-stable.

In the following, we prove  $\mathcal{M}_t$ is vertex-extendable. Let $\epsilon >0$ be a sufficiently small constant and $N_0$ be a sufficiently large integer. Suppose that $\mathcal{H}$ is a  $\mathcal{M}_t$-free $3$-graph on $n\ge N_0$ vertices with
\begin{align}\label{ALIGN:delta(H)}
\delta(\mathcal{H})\ge \left(3\lambda(K_{3t+3}^{3-})-\epsilon\right)n^2.
\end{align}

By Theorem \ref{THM:Psi-trick:G-extendable-implies-degree-stability}, it suffices to show that if there is a vertex $v$ such that $\mathcal{H}'=\mathcal{H}-v$ is $\mathcal{G}_{i_0}$-colorable for some $i_0\in [t]$, then $\mathcal{H}$ is also $\mathcal{G}_{i_0}$-colorable. We only consider the case that  $i_0=1$. The other cases are similar. Let
\[
\mathcal{P}=\{V_1, V_{1'}, V_2, V_{2'},V_3, \ldots, V_{3t+3}\}
\]
be a partition of $V(\mathcal{H}')$ such that every edge in $\mathcal{H}'$ intersects  at most one vertex in each part, and there is no the following four types edges :
\begin{enumerate}
\item [(1)] $E_1=\{xyz : x\in V_i, y\in V_{i'}, z\in V_3, \text{ for } i\in [2]\}$,
\item [(2)] $E_2=\{xyz : x\in V_1, y\in V_2, z\in V_4 \text{ or } x\in V_{1'}, y\in V_{2'}, z\in V_4\}$,
\item [(3)] $E_3=\{xyz : x\in V_1, y\in V_{2'}, z\in V_i, \text{ or } x\in V_{1'}, y\in V_2, z\in V_i \text{ for } i\in[5, 3t+3]\}$, and
\item [(4)] $E_4=\{xyz : x\in V_{3t+1}, y\in V_{3t+2}, z\in V_{3t+3}\}$.
\end{enumerate}
Define $\widehat{\mathcal{G}}_1=\mathcal{G}_1[V_1, V_{1'}, V_2, V_{2'},V_3, \ldots, V_{3t+3}]$ to be the associated blowup of $\mathcal{G}_1$. 
Suppose that
$$Z(\mathcal{G}_1)=\left\{a/2,a/2,a/2,a/2, \underbrace{a,\ldots,a}_{3t-2},b,b,b\right\},$$
where $a=\frac{2}{3t+3+\sqrt{9t^2+18t-3}}$ and $b=\frac{1}{3}-ta$.
Then, as in the proof of Lemma \ref{LEMMA:lambda(K-)-is-irrational}, we have
\begin{align}\label{P-3t+3-function-a}
\lambda(\mathcal{G}_1)=p_{K_{3(t+1)}^{3-}}\left(a/2,a/2,a/2,a/2, \underbrace{a,\ldots,a}_{3t-2},b,b,b\right)=\frac{t}{2}\left(2a^3-(3t+3)a^2+2a\right),
\end{align}
and
\begin{align}\label{a2=0}
6a^2+2-2a(3t+3)=0.
\end{align}

The following claim gives a rough structure of  $\mathcal{H}'$. 
\begin{claim}\label{CLAIM:five-prop}
There exists a constant $c:=c(t)$ such that the following holds.
\begin{itemize}
  \item [(a)] $\left||V_i|-an/2\right|\leq c \epsilon^{1/2} n$ for $i\in \{1,2,1',2'\}$.
  \item [(b)] $||V_i|-an|\leq c \epsilon^{1/2} n$ for $i\in [3, 3t]$.
  \item [(c)] $||V_i|-bn|\leq c \epsilon^{1/2} n$ for $i\in [3t+1, 3t+3]$.
  \item [(d)] If $i\in \{1',2'\}\cup [3t+3]$ and $u\in V(\mathcal{H}')\setminus V_{i}$, then $|V_{i}\setminus N_{\mathcal{H}}(u)|\leq c\epsilon^{1/2} n$.
\end{itemize}
\end{claim}
\begin{proof}[Proof of Claim \ref{CLAIM:five-prop}]
((a)-(c)). For $i\in \{1',2'\}\cup [3t+3]$ let  $x_i=|V_i|/(n-1)$. Let $y_1=y_2=(x_1+x_1'+x_2+x_2')/4$ and $y_i=x_i$ for $i\in [3, 3t]$. 
By \eqref{ALIGN:delta(H)}, we have
\begin{align}\label{ALIGN:number-H'}
|\mathcal{H}'|\ge  \frac{\left(3\lambda(K_{3t+3}^{3-})-\epsilon\right)n^2\times n}{3}-\binom{n}{2}\ge (\lambda(K_{3t+3}^{3-})-\epsilon)n^3,
\end{align}
which means that
\begin{align}\label{ALIGN:larg-H'-small}
p_{\mathcal{G}_1}(x_1,x_{1'},x_2,x_{2'},x_3,\ldots,x_{3t+3})\ge \lambda(K_{3t+3}^{3-})-\epsilon.
\end{align}

On the other hand, let
\[
q=\sum_{i=1}^{3t} y_i \,\,\,\,\,\text{ and }\,\,\,\,\,x=\frac{q}{3t}.
\]
First, we bound $x$. As in the proof Lemma \ref{LEMMA:lambda(K-)-is-irrational} and Theorem \ref{mix-crossed-blowup-lag}, we have
\begin{align}\label{Lag-x<y}
p_{\mathcal{G}_1}(x_1,x_{1'},x_2,x_{2'},x_3,\ldots,x_{3t+3})\leq&p_{K_{3t+3}^{3-}}(y_1,y_2,y_3,\ldots,y_{3t+3})\\
\leq & \frac{t}{2}\left(2 x^3-(3t+3) x^2 +2x \right).\notag
\end{align}
Let $$f(x)=2x^3-(3t+3)x^2+2x.$$
Then $f(x)$ increases monotonically on $[0,a]$, decreases monotonically on $(a,1]$, and $f(a)=2\lambda(K_{3t+3}^{3-})/t$. By \eqref{ALIGN:larg-H'-small}, we have
\begin{align}\label{fa-fx-bound}
f(a)-f(x)\le \frac{2\epsilon}{t}.
\end{align}

Let $x=a+\sigma$, where $\sigma\in [-a, 1-a]$. Then by \eqref{a2=0}, we have
\begin{align*}
f(a)-f(x)&=2a^3-(3t+3)a^2+2a-2(a+\sigma)^3+(3t+3)(a+\sigma)^2-2(a+\sigma)\\
&=-\left(6a^2-2(3t+3)a+2\right)\sigma+(3t+3-6a)\sigma^2-2\sigma^3\\
&=(3t+3-6a)\sigma^2-2\sigma^3\\
&\ge (3t+1-6a)\sigma^2.
\end{align*}
This together with \eqref{fa-fx-bound} yields that
\begin{align}\label{bound-average-a-x-epsilon}
|\sigma|\le \sqrt{\frac{2}{t(3t+1-6a)}}\epsilon^{1/2}<\frac{2}{t}\epsilon^{1/2}.
\end{align}
This implies that
\begin{align}\label{bound-q-epsilon}
\left|q-3ta\right|<6\epsilon^{1/2} \text{ and } |y_{3t+1}+y_{3t+2}+y_{3t+3}-3b|< 6\epsilon^{1/2}.
\end{align}

Next we bound each $y_i$ for $i\in [3t+3]$ through bounding $y_1$ as an example. Recall that $x=\frac{q}{3t}$ and let $y_1=x+(3t-1)\alpha$. Then $\frac{q-y_1}{3t-1}=x-\alpha$.  Let $\beta=\frac{1-q}{3}$.
Note that
\begin{align*}
p_{K_{3t+3}^{3-}}(y_1,y_2,y_3,\ldots,y_{3t+3})=& \sum_{\{i, j, k\}\in \binom{[3t]}{3}}y_iy_jy_k+\left(y_{3t+1}+y_{3t+2}+y_{3t+3}\right)\sum_{\{i, j\}\in \binom{[3t]}{2}}y_iy_j \notag\\
&+(y_{3t+1}y_{3t+2}+y_{3t+1}y_{3t+3}+y_{3t+2}y_{3t+3})\sum_{j\in [3t]}y_j\notag\\
\leq & y_1\binom{3t-1}{2}\left(x-\alpha\right)^2+\binom{3t-1}{3}\left(x-\alpha\right)^3\notag \\
&+(1-q)\left(y_1(q-y_1)+ \binom{3t-1}{2}\left(x-\alpha\right)^2\right)+q\binom{3}{2}\beta^2\\
=&p_{K_{3t+3}^{3-}}\left(y_1,x-\alpha,\ldots,x-\alpha,\beta,\beta,\beta\right).
\end{align*}
On the other hand, we know that
\[
\lambda(K_{3t+3}^{3-})\ge  p_{K_{3t+3}^{3-}}(x,\ldots,x,\beta,\beta,\beta).
\]
So, by \eqref{ALIGN:larg-H'-small}, we have
\begin{align}\label{Pq/3t>Py1}
p_0=p_{K_{3t+3}^{3-}}(x,\ldots,x,\beta,\beta,\beta)- p_{K_{3t+3}^{3-}}(y_1,x-\alpha,\ldots,x-\alpha,\beta,\beta,\beta)\le \epsilon.
\end{align}
Note that 
\begin{align*}
p_0&=\binom{3t}{3}x^3-y_1\binom{3t-1}{2}(x-\alpha)^2-\binom{3t-1}{3}\left(x-\alpha\right)^3   \\
&+(1-q)\left(\binom{3t}{2}x^2-y_1(1-y_1)-\binom{3t-1}{2}(x-\alpha)^2\right). 
\end{align*}
Considering the number of edges of a complete $3t$-partite graph with $qn$ vertices, we have 
\[
y_1(1-y_1)+\binom{3t-1}{2}(x-\alpha)^2-\binom{3t}{3}x^2\leq 0.
\]
So, by  \eqref{Pq/3t>Py1}, we know 
\begin{align}
\binom{3t}{3}x^3-y_1\binom{3t-1}{2}(x-\alpha)^2-\binom{3t-1}{3}\left(x-\alpha\right)^3\leq \epsilon.\notag
\end{align}
Notice that
\begin{align*}
&\binom{3t}{3}x^3-y_1\binom{3t-1}{2}(x-\alpha)^2-\binom{3t-1}{3}\left(x-\alpha\right)^3\\
=& \binom{3t-1}{2}q\alpha^2-2\binom{3t}{3}\alpha^3\\
\geq & \frac{(3t-2)^2}{2}q\alpha^2.
\end{align*}

By \eqref{bound-q-epsilon}, we have
\[
|\alpha|\leq \frac{1}{3t-2}\sqrt{\frac{2}{q}}\epsilon^{1/2}\leq \frac{1}{3t-2}\sqrt{\frac{2}{3ta-6\epsilon^{1/2}}}\epsilon^{1/2}< \frac{1}{3t-2}\sqrt{\frac{2}{(3t-1)a}}\epsilon^{1/2},
\]
which together with \eqref{bound-average-a-x-epsilon} yields that 
\[
|y_1-a|\le |y_1-x|+|x-a|\le \left(\frac{2}{t}+\frac{3t-1}{3t-2}\sqrt{\frac{2}{(3t-1)a}}\right)\epsilon^{1/2}. 
\]

Using  similar arguments, we can show that there exists a constant $c_1:=c_1(t)>0$ such that  $||V_i|-an|\leq c_1 \epsilon^{1/2} n$ for $i\in [1, 3t]$, and  $||V_i|-bn|\leq c_1 \epsilon^{1/2} n$ for $i\in [3t+1, 3t+3]$. This proves (b) and (c).

Now we prove (a).  Let $ x_{1,2}=\frac{x_1+x_1'+x_2+x_2'}{4}$ and 
\[
\psi(i,j, k,\ell)=(x_i+x_j)(x_k+x_{\ell})\,\, \text{ for } \,\,\{i,j,k,\ell\}=\{1,1',2,2'\}.
\] 
By \eqref{ALIGN:larg-H'-small}, we have
\begin{align}\label{upper-bound-1-2-1'-2'}
\epsilon \ge &\lambda(K_{3t+3}^{3-})-p_{\mathcal{G}_1}(x_1,x_{1'},x_2,x_{2'},x_3,\ldots,x_{3t+3})\notag\\
\geq &p_{\mathcal{G}_1}\left(x_{1,2},x_{1,2},x_{1,2},x_{1,2},x_3,\ldots,x_{3t+3}\right)-p_{\mathcal{G}_1}(x_1,x_{1'},x_2,x_{2'},x_3,\ldots,x_{3t+3})\notag\\
= & x_3\left(x_{1,2}^2-\psi(1,1',2,2')\right)+x_4\left(x_{1,2}^2-\psi(1,2,1',2')\right)+ \sum_{i\in[5, 3t+3]}x_i \left(x_{1,2}^2-\psi(1,2',1',2)\right)\notag\\
\geq & \frac{1}{4}\min\left\{x_i : i\in [3, 3t+3]\right\}\left(\sum_{\{i, j\}\in \binom{\{1,1',2,2'\}}{2}}(x_i-x_j)^2\right). 
\end{align}
Set 
\[
\frac{1}{4}\min\{x_i : i\in [3, 3t+3]\}=c_{2}^{1/2}. 
\] 
By \eqref{upper-bound-1-2-1'-2'}, we have $|x_i-x_j|\leq c_2 \epsilon^{1/2}$ for $\{i, j\}\in \binom{\{1, 1', 2, 2'\}}{2}$. Recall that  $|x_1+x_{1'}+x_2+x_{2'}-2a|<2c_1\epsilon^{1/2}$. Thus  
\[
|x_i-a/2|<c_3\epsilon^{1/2}
\]
for $i\in \{1, 1', 2, 2'\}$, where $c_3:=6c_1+c_2$.  This prove (a).

(d). For the proof of \ref{CLAIM:five-prop} (a)--(c), we have that   $||V_i|-\frac{an}{2}|\leq c_3 \epsilon^{1/2} n$ for $i\in \{1,2,1',2'\}$, $||V_i|-an|\leq c_3\epsilon^{1/2} n$ for $i\in [3, 3t]$, and $||V_i|-bn|\leq c_3 \epsilon^{1/2} n$ for $i\in [3t+1, 3t+3]$. Note that if  a 3-graph  $\mathcal{G}$  is a blowup of   $\mathcal{G}_1$ with $|\mathcal{G}|=\lambda(\mathcal{G}_1)v(\mathcal{G})^3$, then $\mathcal{G}$ should be regular.  Recall that $\widehat{\mathcal{G}}_1=\mathcal{G}_1[V_1, V_{1'}, V_2, V_{2'},V_3, \ldots, V_{3t+3}]$ is the associated blowup of $\mathcal{G}_1$. It follows from Claim \ref{CLAIM:five-prop} (a)--(c) that there exists a constant $c_4$ which only depends on $c_3$ such that 
\[
\Delta(\mathcal{H}')\leq \Delta(\widehat{\mathcal{G}}_1)\leq\left(3\lambda(K_{3t+3}^{3-})+c_4\epsilon^{1/2}\right)n^2.
\]

We bound $|V_{2}\setminus N_{\mathcal{H}}(u)|$ for each $u\in V_1$. Let $c_{1,2}=\frac{4c_4}{a}$. If $|V_{2}\setminus N_{\mathcal{H}'}(u)|>c_{1,2}\varepsilon^{1/2}n^2$, then by  Claim \ref{CLAIM:five-prop} (a)--(c), we have 
\begin{align*}
d_{\mathcal{H}'}(u) &\leq \Delta(\widehat{\mathcal{G}}_1)- |V_{2}\setminus N_{\mathcal{H}'}(u)||V_3|\\
& < \left(3\lambda(K_{3t+3}^{3-})+c_4\epsilon^{1/2}\right)n^2-c_{1,2}\epsilon^{1/2}(a-c_3\epsilon^{1/2})n^2\\
& < 3\lambda(K_{3t+3}^{3-})-c_4\epsilon^{1/2}n^2. 
\end{align*}
a contradiction to \eqref{ALIGN:delta(H)}. Thus, $|V_{2}\setminus N_{\mathcal{H}'}(u)|\le c_{1,2}\varepsilon^{1/2}n^2$. Obviously, using the same argument, for $i,j\in \{1',2'\}\cup [3t+3]$ and $i\neq j$, there exists a constant $c_{i,j}$ such that  $|V_{i}\setminus N_{\mathcal{H}}(u)|\leq c_{ij}\epsilon^{1/2} n$ for each  $u\in V(\mathcal{H}')\setminus V_{i}$. 

Set 
$$c^*=\max\{c_{i,j} : i, j\in \{1', 2',\}\cup [3t+3]\,\, \text{and}\,\, i\neq j\}$$ 
and $c=\max\{c_1, c_3, c^*\}$. Then (a)--(d) hold for $c$. This completes the proof of Claim \ref{CLAIM:five-prop}. 
\end{proof}


%

Recall that $\mathcal{H}'$ is $(3t+5)$-partite with the vertex partition $\mathcal{P}$. The following claim shows that $L_{\mathcal{H}}(v)$ is also $(3t+5)$-partite.

\begin{claim}\label{CLAIM:|E-CAP-V_i|<2}
For each $E\in\mathcal{H}$ with $v\in E$, we have $|E\cap V_{i'}|\leq 1$ for $i\in [2]$ and $|E\cap V_i|\leq 1$ for $i\in [3t+3]$.
\end{claim}
\begin{proof}[Proof of Claim \ref{CLAIM:|E-CAP-V_i|<2}]
We only consider the case that $|E\cap V_1|\leq 1$. The other cases are
 similar. By contradiction, suppose that $\{u_1, u'_1\}\subseteq E\cap V_1$. Let $V_i'=N_{\mathcal{H}'}(u_1)\cap  N_{\mathcal{H}'}(u_1')\cap V_{i}$ for $i\in \{1', 2'\}\cup[2, 3t+3]$. By Claim \ref{CLAIM:five-prop} (a)--(d), we have the following holds:
\begin{itemize}
  \item  $|V_{i}'|\geq an/5$ for $i\in \{1', 2, 2'\}$, 
  \item  $|V_i'|\geq an/3$ for $i\in [3, 3t]$, and
  \item  $|V_i'|\geq bn/3$ for $i\in [3t+1, 3t+3]$.
\end{itemize}

Let $W=\{u_1,u_1'\}\cup V_{1'}'\cup V_{2'}'\cup (\cup_{i=2}^{3t+3}V_i')$. We can apply Lemma \ref{LEMMA:greedily-embedding-Gi} to the 3-graph $\mathcal{H}[W]$ with $S=\{u_1, u'_1\}$, $T=\{1', 2'\}\cup [2, 3t+3]$ and $\eta=c\epsilon^{1/4}$. There exists  $u_j\in V'_j$ for $j\in \{1', 2'\}\cup[2, 3t+3]$ such that both $\mathcal{H}[U\cup \{u_1\}]$ and $\mathcal{H}[U\cup \{u'_1\}]$ are copies of $\mathcal{G}_1$, where $U=\{u_j : j\in \{1', 2'\}\cup[2, 3t+3]\}$. This implies that $U\cup \{u_1, u_1'\}$ is a  2-covered set of $\mathcal{H}[U\cup \{u_1,u'_1, v\}]$. Recall that $\{ u_1,u'_1,v\}\in \mathcal{H}$.  Thus, $\mathcal{H}[U\cup \{u_1,u'_1, v\}]\in \mathcal{K}_{3t+6}^3$, a contradiction to the fact  $\mathcal{K}_{3t+6}^3\subseteq \mathcal{M}_t$.
\end{proof}

We give a lower bound on $L_{\mathcal{G}_1}(v)$ with respect to $a,b$. 
Let
\[
\mathcal{U}=\{U_1, U_{1'}, U_2, U_{2'},U_3, \ldots, U_{3t+3}\}
\]
be a partition of $V(\mathcal{H})$ with $|U_i|=a/2$ for $i\in \{1,2,1',2'\}$, $|U_i|=a$ for $3\le i\le 3t$ and $|U_i|=b$ for $3t+1\le i\le 3t+3$. Let  $\widetilde{\mathcal{G}}_1=\mathcal{G}_1[U_1, U_{1'}, U_2, U_{2'},U_3, \ldots, U_{3t+3}]$ to be the associated blowup of $\mathcal{G}_1$. Then $|\widetilde{\mathcal{G}}_1|=\lambda(\mathcal{G}_1)n^3$. As mentioned in \cite{LMR2}, $\widetilde{\mathcal{G}}_1$ should be almost regular in the sense that $|d(x)-d(y)|=o(n^2)$ for any $x,y\in V(\widetilde{\mathcal{G}}_1)$. For $u\in U_1$, 
\begin{align*}
    |L_{\mathcal{G}_0}(u)|=& \sum_{\{i, j\}\in \binom{[2,3t+2]}{2}}|U_i||U_j|+\sum_{\{i, j\}\in \binom{[3t+1,3t+3]}{2}}|U_i||U_j|\\
    &+(|U_{1'}|+|U_{2}|)|U_3|+(|U_2|+|U_{2'}|)|U_4|+(|U_{1'}|+|U_2|)\left(\sum_{i=5}^{3t+3}|U_i|\right)\\
    =& \binom{3t-2}{2}a^2n^2+3b^2n^2+3(3t-2)abn^2+(3t-2)a^2n^2+3abn^2\\
    =&\left(\binom{3t-1}{2}a^2+3b^2+3(3t-1)ab\right)n^2.
\end{align*}
Thus, 
\[
3\lambda(K_{3t+3}^{3-})=\binom{3t-1}{2}a^2+3b^2+3(3t-1)ab. 
\]
By \eqref{ALIGN:delta(H)}, we have 
\begin{align}\label{lower-bound-link-v-mini}
L_{\mathcal{H}}(v)\ge 3\lambda(K_{3t+3}^{3-})=\left(\binom{3t-1}{2}a^2+3b^2+3(3t-1)ab-\epsilon\right)n^2. 
\end{align}

Then, we consider the neighbors of $v$ in $\mathcal{H}$. For simplicity, let $V_1^*=V_1\cup V_{1'}$, $V_2^*=V_2\cup V_{2'}$ and $V_i^*=V_i$ for $i\in [3, 3t+3]$.  Set
\[
V'_i=N_{\mathcal{H}}(v)\cap V^*_i. 
\]
for $i\in [3t+3]$.

We claim that there are no two distinct index $i, j\in [3t+3]$ such that both $\max\{|V'_i|, |V'_j|\}\leq |V_i^*|/5$. Indeed, suppose to the contrary that there are two sets $V_i'$ and $V_j'$ having size at most $|V_i^*|/5$ and $|V_j^*|/5$, respectively, say $i=1$ and $j=2$. Let 
\[
G=K_{3t+5}[V'_1\cap V_1, V'_1\cap V_{1'},V'_2\cap V_2,V'_2\cap V_{2'},V'_3,\ldots V'_{3t+3}]
\] 
be the associated blowup of $K_{3t+5}$. Then $ |L_{\mathcal{H}}(v)|\le |G|$. Now we count the number of edges in $G$. By Claim \ref{CLAIM:five-prop}, $|V'_i|\le (a+ 2c \epsilon)n/5$ for $i\in\{1,2\}$.  Clearly, $|V'_i|\le |V_i|$. So, using Claim \ref{CLAIM:five-prop} again, we have 
\begin{align*}
    |L_{\mathcal{H}}(v)|
    \leq& \binom{3t-2}{2}a^2n^2 +\binom{3}{2}b^2n^2 +3(3t-2)abn^2 + \left(\frac{an}{5}\right)^2+2\left(\frac{an}{10}\right)^2\notag\\
    +& \left(\frac{2an}{5}\right)\left((3t-2)an+3bn \right)+C \epsilon^{1/4}n^2\notag\\
    =&\left( \binom{3t-2}{2}a^2+3b^2n^2+3(3t-2)ab+\frac{(60t-37)a^2}{50}+\frac{6ab}{5}+C \epsilon^{1/4}\right)n^2\notag\\
    =&\left(\binom{3t-1}{2}a^2+3b^2+3(3t-1)ab-\frac{90t-63}{50}a^2-\frac{8}{5}ab+C \epsilon^{1/4}\right)n^2, \notag\\
\end{align*}
where $C$  is a constant   depends on $t,a,b,c$. Thus, we get a contradiction to \eqref{lower-bound-link-v-mini}.

For simplicity, define 
\begin{itemize}
\item $I_{small}=\{i\in \{1', 2'\}\cup [3t+3] : |N_{\mathcal{H}}(v)\cap V_i|\leq \frac{|V_i|}{10}\}$, and
\item $I_{big}=(\{1', 2'\}\cup [3t+3] ) \setminus I_{small}.$
\end{itemize}
By the above arguments, we know that  $|I_{small}|\le 3$.  In the following, we show that there exists some $i\in I_{small}$ and $u\in V_i$ such that $L_{\mathcal{H}}(v)\subseteq L_{\widehat{\mathcal{G}}_1}(u)$. This implies that $\mathcal{H}$ is also $\mathcal{G}_1$-colorable, which completes the proof.

First, considering the codegree of pairs in $\mathcal{G}_1$,  we have the following observation.
\begin{observation}\label{OBSERVATION:D_ij}
\[
d_{\mathcal{G}_1}(i,j)=
\begin{cases}
  2, & \text{if } \,\, \{i, j\}\in \{\{1,2'\},\{1',2\}\},\\
  3t, & \text{if } \,\,\{i, j\}\in \{1, 2'\}\times \{1', 2\}, \\
  3t+2, & \text{if }\,\, \{i, j\}\in \{1,1',2,2'\}\times [3, 3t+3], \\
  3t+2, & \text{if }\,\, \{i, j\}\in\binom{[3t+1, 3t+3]}{2},  \text{ and } \\
  3t+3, & \text{if }\,\, \{i, j\}\in\binom{[3, 3t+3]}{2}-\binom{[3t+1, 3t+3]}{2}.
\end{cases}
\]
\end{observation}

Then,  we construct an auxiliary graph $R$ on the vertex set $\{1', 2'\}\cup [3t+3]$ in which  $i, j$ are adjacent if and only if there exists $e\in L_{\mathcal{H}}(v)$ such that $|e\cap V_i|=|e\cap V_j|=1$. For each edge $ij\in R$, we choice an edge  $e_{ij} \in L_{\mathcal{H}}(v)$  with one endpoint in $V_i$ and the other endpoint in $V_j$. Let $U_1 =\cup_{ij \in R}e_{ij}$. Then  $|U_1| \leq \binom{3t+5}{2}$.
Let 
\[
V'_i=
\begin{cases}
  V_i\cap N_{\mathcal{H}}(v)\cap \left(\cap_{u\in U_1\setminus V_i}N_{\mathcal{H}'}(u) \right)  & \mbox{\text{for} } i\in I_{big},\\
  V_i\cap \left(\cap_{u\in U_1\setminus V_i}N_{\mathcal{H}'}(u) \right)  & \mbox{\text{otherwise}}.
\end{cases}
\]
It follows from Claim \ref{CLAIM:five-prop} that 
\[
|V_i'|\geq \min\left\{\frac{a}{20}n, \frac{b}{10}n\right\}-\frac{1}{10}c\epsilon^{1/2}n-|U_1|c\epsilon^{1/2}n\ge \frac{b}{30}n.
\]
Applying Lemma \ref{LEMMA:greedily-embedding-Gi} to $\mathcal{H}'[\cup_{i\in \{1', 2'\} \cup[3t+3]}V_i']$ with $S=U_1$, $T=\{1', 2'\}\cup [3t+3]$ and $\eta=c\epsilon^{1/2}$. There exists $u_j\in V_j'$ for $\{1', 2'\}\cup [3t+3]$ such that the following statements hold by letting  $U_2=\{u_i : i\in \{1', 2'\}\cup [3t+3]\}$. 
\begin{itemize}
  \item $\mathcal{H}'[U_2]\cong \mathcal{G}_1$, and
  \item $L_{\mathcal{H}'}(u)[U_2]=L_{\widehat{\mathcal{G}}_1}(u_i)[U_2]$ for all $i\in \{1' ,2'\}\cup [3t+3]$ and $u\in U_1\cap V_i$.
\end{itemize}

Finally, let $F=\mathcal{H}[\{v\}\cup U_1\cup U_2]$. Then 
\[
|V(F)|\leq 1+|U_1|+|U_2|\leq 1+\binom{3t+5}{2}+3t+5<(3t+6)^2. 
\]
By the construction of $\mathcal{M}_t$, we have $F\notin \mathcal{M}_t$. Thus, there exists a homomorphism $\varphi: V(F)\rightarrow V(\mathcal{G}_1)$ from $F$ to $\mathcal{G}_1$. Note that $F[U_2]\cong \mathcal{G}_1$. Considering the codegree of pairs in $F[U_2]$ and by Observation \ref{OBSERVATION:D_ij}, we have $\varphi |_{U_2}$ is bijective, and 
\begin{itemize}
  \item $\varphi(\{u_1, u_{1'}. u_2, u_{2'}\})=\{1, 1', 2, 2' \}$,
  \item $\varphi(\{u_i : i\in [3, 3t+3]\})= [3, 3t+3]$, and
  \item $\varphi(\{u_i : i\in [3+1, 3t+3]\})= [3t+1, 3t+3]$.
  \end{itemize}
By symmetry we may suppose that $\varphi(u_i)=i$ for $i\in \{1', 2'\}\cup [3t+3] $. For $u\in V_i\cap U_1$ we have $\varphi(u)=i$ by $L_{\mathcal{H}'}(u)[U_2]=L_{\widehat{\mathcal{G}}_1}(u_i)[U_2]$. Since $v$ is adjacent to all vertices in $\{u_i : i\in I_{big}\}$  and $\varphi$ is a homomorphism, we have $\varphi(v)\in I_{small}$. This means that there exists some $i_0\in I_{small}$ such that $\varphi(R)\subseteq L_{\mathcal{G}_1}(i_0)$, i.e., $L_{\mathcal{H}}(v)\subseteq L_{\widehat{\mathcal{G}}_1}(w)$ for $w\in V_{i_0}$. This completes the proof of Theorem \ref{THM:main-sec1.1} (b). 
\end{proof}

To end this section we prove Theorem \ref{THM:main-sec1.1} (c) using (b).

\begin{proof}[Proof of Theorem \ref{THM:main-sec1.1} (c)]
Recall that for $i\in [t]$, $\mathcal{G}_n^i$ is a 3-graph on $n$ vertices which is a blowup of $\mathcal{G}_i$ with the maximum number of edges. We need to  add or remove $\Omega(n^3)$ edges to transfer $\mathcal{G}_n^i$  to $\mathcal{G}_n^j$ for $1\le i<j\le t$. This together with each $\mathcal{G}_i$ is $\mathcal{M}_t$-free yields that $\xi(\mathcal{M}_t)\ge t$.

For the other side, for sufficiently small $\epsilon$ and large $n$, suppose that  $\mathcal{H}$ is an $\mathcal{M}_t$-free $r$-graph on $n$ vertices with $|\mathcal{H}| > (1-\epsilon) \mathrm{ex}(n,\mathcal{M}_t)$. Let 
    \begin{align*}
        Z_{\varepsilon}(\mathcal{H}) := \left\{v\in V(\mathcal{H}) \colon d_{\mathcal{H}}(v) \le \left(\pi(F)-2\varepsilon^{1/2}\right)\binom{n-1}{r-1}\right\}. 
    \end{align*}
Using  similar arguments in~{\cite[Lemma~4.2]{LMR1}} with some minor modifications, we have that 
    \begin{enumerate}[label=(\alph*)]
    \item
    the set $Z_{\varepsilon}(\mathcal{H})$ has size at most $\varepsilon^{1/2}n$, and
    \item
    the induced subgraph $\mathcal{H}'$ of $\mathcal{H}$ on $V(\mathcal{H})\setminus Z_{\varepsilon}(\mathcal{H})$ satisfies $\delta(\mathcal{H}') \ge \left(\pi(F)-3\varepsilon^{1/2}\right)\binom{n-1}{r-1}$.
    \end{enumerate}
Applying Theorem \ref{THM:main-sec1.1} (b) to $\mathcal{H}'$  yields that $\xi(\mathcal{M}_t)\le t$. 
\end{proof}

\section{Concluding remarks}\label{SEC:Remarks}

In this paper,   we construct a  finite family  $\mathcal{M}_t$ of 3-graphs such that the Tur\'{a}n density of $\mathcal{M}_t$ is  irrational,  and there are $t$ near-extremal $\mathcal{M}_t$-free configurations that are far from each other in edit-distance.  This is the first not stable example that has an irrational  Tur\'{a}n density. Unlike the construction given  in~\cite{LMR1} with  the help of algebraic methods,  we define an operation on 3-graphs. It is more intuitive and easier to understand.

The construction of  $\mathcal{M}_t$ also provides a new phenomenon about feasible region functions. Let $\mathcal{H}$ be an $r$-graph. Recall that  the shadow of $\mathcal{H}$ is 
\begin{align}
\partial\mathcal{H}
=
\left\{A\in \binom{V(\mathcal{H})}{r-1}\colon \text{there is } B\in \mathcal{H} \text{ such that }
	A\subseteq B\right\}. \notag
\end{align}
The  classical Kruskal--Katona theorem~\cite{KA68,KR63} gives a tight lower bound on $|\partial\mathcal{H}|$ with respective to $|\mathcal{H}|$. Liu and Mubayi ~\cite{LM1} extended it and defined the feasible region $\Omega(\mathcal{F})$  and feasible region function $g(\mathcal{F})$ of a family $\mathcal{F}$ of $r$-graphs\footnote{We omit their formal definitions and refer the reader to \cite{LM1} for  detail.}. The global maxima of $g(\mathcal{F})$ is exactly the Tur\'{a}n density of $\mathcal{F}$, and the stability number of $\mathcal{F}$ is closely related to  the number of global maxima of $g(\mathcal{F})$ (e.g. see~\cite{LMR1}). So, $g(\mathcal{F})$ can understand the extremal properties of $\mathcal{F}$-free hypergraphs beyond just the determination of $\pi(\mathcal{F})$. In previous results in \cite{hou2022hypergraphs, LM22, LMR1,LP22}, a family $\mathcal{F}$ with stability number $t$ (or infinity) has $t$ (or infinity) global maxima of $g(\mathcal{F})$. However,  the  shadow of the 3-graph $\mathcal{G}_n^i$ used in our proof is complete $(3t+5)$-partite graph for each $i\in [n]$ and all  shadows of $\mathcal{G}_n^i$ have  same edge densities. This means  $g(\mathcal{M}_t)$ has exactly one  global maxima. This new phenomenon shows that it does exist a not stable finite family of $r$-graphs with one  global maxima of its feasible region function.  This is perhaps helpful to understand the Tur\'{a}n's Tetrahedron conjecture as the shadow of all  constructions given by  Kostochka~\cite{KO82} is complete.

\bibliographystyle{abbrv}
\bibliography{t_stable}
\end{document}